\documentclass{siamltex1213}

\usepackage{graphics}
\usepackage[leqno]{amsmath}
\usepackage{amssymb}
 \usepackage{multirow}
\usepackage{pstricks}
\usepackage{graphicx,subfig,epstopdf,epsfig,floatrow}
\usepackage[linesnumbered,ruled,vlined]{algorithm2e}

\title{Splitting and successively solving augmented Lagrangian method for optimization with semicontinuous variables and cardinality constraint\thanks{This research was supported by a grant from the National Natural Science Foundation of China (No.11371242)}}

\author{Yanqin Bai\thanks{Department of Mathematics, Shanghai University, Shanghai,200444, China.
(\email{yqbai@shu.edu.cn})}
\and Renli Liang\thanks{Department of Mathematics, Shanghai University, Shanghai, 200444, China.
(\email{13720011@shu.edu.cn})}
\and Zhouwang Yang\thanks{School of Mathematical Sciences, University of Science and Technology of China, Hefei, 230026, China
(\email{yangzw@ustc.edu.cn })}
}

\begin{document}
\maketitle
\slugger{simax}{xxxx}{xx}{x}{x--x}
\begin{abstract}
We propose a new splitting and successively solving augmented Lagrangian (SSAL) method for solving an optimization problem with both semicontinuous variables and a cardinality constraint. This optimization problem arises in several contexts such as the portfolio selection problem, the compressed sensing problem and the unit commitment problem, etc.
The problem is in general NP-hard. We derive an optimality condition for this optimization problem, under some suitable assumptions.
By introducing an auxiliary variable and using an augmented Lagrangian function, the constraints are decomposed into two parts.
By fixing particular variables, the optimization problem is split into two subproblems.
The first one is an easy solving convex programming problem, while the second one is a complicated quadratic programming problem with semicontinuous variables and cardinality constraint. The two subproblems are solved alternatively.
The second subproblem is transformed equivalently into a mixed integer quadratic programming problem.
Based on its features of both the objective function and the constraints, we solve the mixed integer quadratic programming problem successively, then obtain its exact optimal solution directly. Furthermore, we prove the convergence of SSAL, under some suitable assumptions.
Finally, we implement our method for the portfolio selection problem and the compressed sensing problem, respectively.
Real world data and simulation study show that SSAL outperforms the well-known CPLEX 12.6 and the penalty decomposition (PD) method, while enjoying a similar cardinality of decision variable. For example, the numerical results are even nearly 200 times faster than that of CPLEX 12.6 for the portfolio selection problem, and more than 40 times faster than PD method for the compressed sensing problem, respectively. In particularly, SSAL is powerful when the size of problem increase largely.
\end{abstract}


%

\section{Introduction} \label{se:introduction}
In this paper we consider the following minimization problem:
\begin{eqnarray*}
\rm{(P)}~~~~\min  ~~ &&f(x) \\
 {\rm s.t.}~~~&&x\in X,\\
 && \|x\|_{0}\leq K,\\
 && x_{i}\in \{0\}\cup[a_{i},b_{i}],\quad i\in\{1,2,\dots,n\},
\end{eqnarray*}
where $x\in R^{n}$, $f(x)$: $R^{n}\rightarrow R$ is a continuously differentiable convex function and  $X\in R^{n}$ is a closed convex set with no empty interior, which represents the general constraints for $x$. $\|x\|_{0}$ denotes the cardinality (i.e., the number of nonzero entries) of the vector $x$. Furthermore $0<K<n$, $0<a_{i}< b_{i}$.
A variable $x_{i}\in \{0\}\cup[a_{i},b_{i}]$ is referred to as a semicontinuous variable.

The problem (P) arises in several contexts such as the portfolio selection problem (see \cite{Bienstock96}, \cite{cui}, \cite{Gaojianjun2013}), the compressed sensing problem (see \cite{Bersimas09}, \cite{Bonami09}), the separable quadratic facility location problem \cite{semi} and the unite commitment problem (see \cite{Frangioni06}), ect. Obviously, the constraints of (P) are both nonconvex and nonsmooth. Problem (P) is a highly nonlinear optimization problems and may have many local minimizers. So, it is difficult to solve problem (P), even if to find its local solutions. In general, problem (P) is NP hard \cite{Bienstock96}. For decades, several analytical and approximate methods, including reformulation, approximation and relaxation, have been developed for solving problem (P).

Padberg and Rinaldi \cite{Padberg:1991} firstly proposed a branch-and-cut algorithm for the symmetric traveling salesman problem, which is an integer linear programming problem.
Frangioni and Gentile \cite{Frangioni06} developed the branch-and-cut algorithm for a general program with a convex objective function and the semicontinuous constraint, by using a perspective function to derive lower bounds on the objective function value.
Zheng, Sun and Li devolved the SDP approach by applying a special Lagrangian decomposition scheme for a quadratic programming problem with the cardinality and the semicontinuous constraints, which gives the tightest continuous relaxation of the perspective reformulation of the problem in \cite{Zheng14}.
The above exact methods are able to solve only a fraction of practically useful models, a variety of heuristic procedures have also been proposed.
Chang et al. presented three heuristic algorithms based
upon genetic algorithms, tabu search and simulated annealing for the standard
mean variance portfolio optimisation model with cardinality constrained in \cite{chang00}.
There are also various approximate methods and techniques for dealing with the cardinality or sparse constraint.
Burdakov et al presented a reformulation for the cardinality-constrained problems, including the semi-continuous case and they formulated the problems into a mixed-integer ones in literature \cite{Burdakov14}. Moreover, they also introduced a mixed-integer formulation whose standard relaxation still has the same solutions (in the sense of global minima)
as the underlying cardinality-constrained problem \cite{Burdakov15}. Furthermore,
they derived the suitable stationarity conditions and suggested an appropriate
regularization method for the solution of optimization problems with cardinality
constraints.
Bai et al proposed an alternating direction method of multipliers (ADMM) method  for $\ell_{1}-\ell_{2}$ regularized logistic regression model in \cite{Bai-Shen-orsc}. And Bai and Shen  developed this method for the consensus proximal support vector machine for classification in \cite{Bai3}.

Recently,
Lu and Zhang \cite{penal} proposed a novel penalty
decomposition (PD) method for an $l_{p}$ norm minimization problem which minimizes a general nonlinear convex function subject to $\|x\|_{0} \leq K$. Besides the cardinality constraints, Cui et al. \cite{cui} investigated a portfolio selection problem with both the cardinality constraints and the semicontinuous constraints, which was reformulated as a mixed integer quadratically constrained quadratic program (MIQCQP).
They devolved the branch-and-bound method by reformulating MIQCQP to a new MIQCQP reformulation, which has a more tighter continuous relaxation bound than the MIQCQP.
They reported the computational results for problems with up to 400 assets.
Although the speed of the PD method is generally faster than the hard-thresholding algorithm \cite{Blumensath08}, the solution obtained by the PD method is not better enough for the solution quality.
Besides, the reformulation of Cui et al. \cite{cui} has more constraints and variables than MIQCQP and caused the difficulty that the branch-and-bound method is slow and fails to solve large-scale problems.

Inspired by requirements of both large-scale problems and the better solution quality,
we propose a new method for solving problem (P), which covers the portfolio selection problem and the compressed sensing problem as special cases. Under mild assumptions,we derive the optimality conditions of problem (P).
By introducing an auxiliary variable and using an augmented Lagrangian function, the constraints are separated into two parts.
By fixing particular variables, the optimization problem is split into two subproblems.
The first one is an easy to solve convex programming problem, while the second one is a complicated quadratic programming problem with semicontinuous variables and cardinality constraint.
Therefore, the second subproblem is transformed into an equivalent  mixed integer quadratic programming problem with $0-1$ knapsack constraint.
Then we propose the augmented Lagrangian alternating direction method and solve the two subproblems alternatively.
The second one can be solved analytically, due to features of both the objective function and the constraints. 
Furthermore, we prove that our method is convergent.
Finally, we implement our method for the portfolio selection problem and the compressed sensing problem.
The numerical results demonstrate that our method is faster than several existing methods. For example, in case of the compressed sensing problem, our method is more than 40 times faster than that of the penalty decomposition (PD) method in \cite{penal}.

The paper is organized as follows.
In Section \ref{sec:Examples of Applications}, we describe some examples of problem (P) arising from different real-world applications.
In Section \ref{sec:The first order optimality conditions}, we establish the first order optimality conditions for the problem (P).
In Section \ref{sec:Alternating direction method} we present our method for solving problem (P).
We name it the Splitting and Successively solving Augmented Lagrangian method (abbreviated as SSAL).
We also establish some convergence results.
In Section \ref{sec:Numerical results}, we conduct numerical experiments to test the performance of SSAL for the portfolio selection problem in \cite{cui} and compressed sensing problems.
Finally, we present some conclusions in Section \ref{sec:Conclusion}.

\subsection{Notation}\label{subsec:Notation}

In this paper, the symbols $R^{n}$ and $R^{n}_{+}$ denote the n-dimensional Euclidean space and the nonnegative orthant of $R^{n}$, respectively.
For any vectors $x=[x_{1},x_{2},\dots,x_{n}]^{T}$, $y=[y_{1},y_{2},\dots,y_{n}]^{T}$,
$x\cdot y=[x_{1}y_{1},x_{2}y_{2},\dots,x_{n}y_{n}]^{T}$, and
$x^{2}=x\cdot x$.
Given an index set $L\subseteq\{1,2,\dots,n\}$ , $|L|$ denotes the size of $L$, $\bar{L}$ is the complement of $L$ in $\{1,2,\dots,n\}$.
$x_{L}$ denotes the subvector formed by the entries of $x$ indexed by $L$. Likewise, $X_{L}$ denotes the submatrix formed by the columns of $X$ indexed by $L$. We denote by I the identity matrix, whose dimension should be clear from the context. Given a closed set $C\subseteq R^{n}$, $N_{C}(x)$ and $T_{C}(x)$ denote the normal and tangent cone of $C$ at any $x\in C$, respectively. For any real vector, $\|\cdot\|_{0}$ and $\|\cdot\|$ denote the cardinality (i.e., the number of nonzero entries) and the Euclidean norm of the vector, respectively.

\section{Examples of Applications}\label{sec:Examples of Applications}
In this section, we describe some examples of problem (P) arising from the portfolio
selection, the compressed sensing and the subset selection. More applications can be found in \cite{semi}.

\subsection{Portfolio selection problem}
Let $\mu$ and $Q$ be the mean vector and the covariance matrix of $n$ assets with return $r=(r_{1},r_{2}\cdots r_{n})^{T}$, respectively. Let
$x=(x_{1},x_{2}\cdots x_{n})^{T}$ be the portfolio weights to the $n$ risky assets. 
In real-world, most investors only choose a small number of stocks to invest. In other words, we need to consider the constraint $\|x\|_{0}\leq K$ to control the total number of different assets in the optimal portfolio.  We also need to prevent the investors from holding some assets with very small amount due to the transaction cost and managerial concerns. So we need to consider the semicontinuous constraints: $x_{i}\in \{0\}\cup[a_{i},b_{i}], i\in \{1,2,\dots,n\}$. A mean-variance portfolio selection model with semicontinuous and cardinality constraints can be modeled as:
\begin{eqnarray*}
{\rm(SP)}~~~~\min~~ &&x^{T}Qx     \\
 {\rm s.t.}~ ~~ &&\mu^{T}x \geq\rho_{0},    \\
&& e^{T}x=1,    \\
&&  \|x\|_{0}\leq K,  \\
 &&  x_{i}\in \{0\}\cup[a_{i},b_{i}],\quad i\in \{1,2,\dots,n\}.
\end{eqnarray*}
where $e=(1,1\cdots1)^{T}$, $\rho_{0}$ is the prescribed return level.
\subsection{Compressed sensing problem}
Compressed sensing (CS) is an important problem in signal processing (see, for example, \cite{S.Chen98}).The CS problem with nonnegativity constraints can be formulated as
\begin{eqnarray*}
{\rm (NCS)}~~~~\min ~~&& \frac{1}{2}\|Ax-b\|^{2}\\
 {\rm s.t.}~ ~~ &&\|x\|_{0}\leq K,\quad x\geq0,
\end{eqnarray*}
where $A\in R^{p\times n}$ is a data matrix, $b\in R^{p}$ is an observation vector, $K$ is an integer for controlling the sparsity of the solution. More applications and reformulations of problem (NCS) can be found in \cite{O'Grady08} and \cite{Khajehnejad11}. In the compressed sensing problem, it is often assumed that $p<n$.

In multivariate linear regression, we are given $p$ observed data points $(a_{i}, b_{i})$, with $a_{i}\in R^{n}$ and $b_{i}\in R$.
The goal is to minimize the least square measure of $\sum_{i=1}^{p}(a_{i}^{T}x-b_{i})^{2}$ with only a subset of the prediction variables in $x$. This subset selection problem then has the same form as the problem (NCS).
In contrast with the case of compressed sensing, the number of data in subset selection is often much larger than the dimension of the data $(p>n)$.

In practice, we can always impose an upper bound on $x$, i.e. $ x\leq u$, for some sufficiently large positive numbers $u$ (see \cite{semi}).
\begin{eqnarray*}
{\rm (NCSB)}~~~~\min ~~&&\frac{1}{2}\|Ax-b\|^{2}\\
{\rm s.t.}~ ~~ &&\|x\|_{0}\leq K,\quad 0\leq x\leq u.
\end{eqnarray*}

%
\section{The first order stationary conditions}
\label{sec:The first order optimality conditions}
In this section, we study the first order stationary conditions of the problem (P).
\begin{theorem}\label{necessary optimality conditions}
Assume that $x^{*}$ is a local minimizer of problem $(P)$. Let $J^{*}=\{1\leq j\leq n:x^{*}_{j} \neq 0\}$, $r^{*}=|J^{*}|$. Suppose that the following Robinson condition
\begin{equation}\label{Robinson}
\begin{split}
\left\{\left(
  \begin{array}{c}
d_{J^{*}}-s\\
-d_{J^{*}}-t\\
d_{\bar{J}^{*}}
  \end{array}
\right)
\Bigg| \begin{array}{l}  d\in T_{X}(x^{*}), s_{A(x^{*})}\leq0, t_{B(x^{*})}\leq0
\end{array}
\right\}=R^{n+r^{*}}
\end{split}
\end{equation}
holds,
where
\begin{equation*}
 A(x^{*})=\{ i\in J^{*}:x_{i}^{*}=b_{i}    \}, \quad
 B(x^{*})=\{ i\in J^{*}:x_{i}^{*}=a_{i}    \}.
\end{equation*}
Then there exists
$(\nu^{*},\eta^{*},\kappa^{*})\in  R^{n}\times R^{n}\times R^{n}$ together with $x^{*}$ satisfying
\begin{equation}\label{opt}
\begin{split}
0\in \nabla f(x^{*})+\nu^{*}-\eta^{*}+\kappa^{*}+N_{X}(x^{*}),\\
\nu_{i}^{*}\geq 0,\quad \eta_{i}^{*}\geq 0,\quad \nu^{*}_{i} (x^{*}_{i}-b_{i})=0,\quad \eta^{*}_{i}(-a_{i}+ x_{i}^{*})=0,\quad i=1,\dots,n,\\
\kappa_{j}^{*}=0,\quad j\in J^{*}.
\end{split}
\end{equation}
\end{theorem}
\begin{proof}
The problem (P) can be reformulated to the following problem:
\begin{eqnarray*}
{\rm (P)}~~~~\min  ~~ &&f(x) \\
 {\rm s.t.}~ ~~&&x\in X,\\
&& a_{i}\leq x_{i}\leq b_{i}, \quad i\in supp(x),\\
 &&  x_{i}=0 ,\quad i\in \overline{supp}(x),\\
 &&  |supp(x)|\leq K,
\end{eqnarray*}
where $supp(x)=\{1\leq i\leq n~|~x_{i}\neq 0\}$, and $\overline{supp}(x)$ is the complement of $supp(x)$ in $\{1,2,\dots,n\}$.
By the assumption that $x^{*}$ is a local minimizer of the problem (P), one can observe that $x^{*}$ is also a local minimizer of the following problem:
\begin{eqnarray*}
{\rm (\overline{P})}~~~~\min  ~~ &&f(x) \\
 {\rm s.t.}~ ~~&&x\in X,\\
&&a_{i}\leq x_{i}\leq b_{i},\quad i\in J^{*},\\
 &&  x_{j}=0,\quad j\in \bar{J}^{*}.
\end{eqnarray*}
We observe that the problem ($\bar{P}$) can be reformulated equivalently as
\begin{eqnarray*}
{\rm(\overline{P1})}~~~~\min  ~~&& f(x) \\
 {\rm s.t.}~ ~~&&x\in X,\\
&& I^{T}_{J^{*}}(x-b)\leq 0,\\
 &&  I^{T}_{J^{*}}(-x+a)\leq 0,\\
  &&I^{T}_{\bar{J}^{*}}x= 0.
\end{eqnarray*}
Using theorem 3.25 of \cite{A.Ruszcunski}, when the Robinson condition $(\ref{Robinson})$ holds, there exists
$(\nu^{*},\eta^{*},\kappa^{*})\in  R^{n}\times R^{n}\times R^{n}$ together with $x^{*}$ satisfying
 \begin{eqnarray}
0\in \nabla f(x^{*})+\nu^{*}-\eta^{*}+\kappa^{*}+N_{X}(x^{*}),\\
\nu_{i}^{*}\geq 0,\quad\eta_{i}^{*}\geq 0,\quad\nu^{*}_{i} (x^{*}_{i}-b_{i})=0,\quad\eta^{*}_{i}(-a_{i}+ x_{i}^{*})=0,\quad i\in J^{*} \label{k2},\\
\nu^{*}_{i}=0,\quad \eta^{*}_{i}=0,\quad i\in \bar{J^{*}}\label{k3},\\
\kappa_{j}^{*}=0,\quad j\in J^{*}.
 \end{eqnarray}
Note that $(\ref{k2})$ implies that if $x_{i}^{*}\neq0$ and $x_{i}^{*}\neq b_{i}$ , then $\nu^{*}_{i}=0$. If $x_{i}^{*}\neq0$ and $y_{i}^{*}\neq a_{i}$, then $\eta^{*}_{i}=0$. The equation $(\ref{k3})$ implies that if $x_{i}^{*}=0$ , then $\nu^{*}_{i}=0$ and $\eta^{*}_{i}=0$. Combining $(\ref{k2})$ with $(\ref{k3})$, we obtain
 \begin{equation*}
 \nu_{i}^{*}\geq 0,\quad\eta_{i}^{*}\geq 0,\quad\nu^{*}_{i} (x^{*}_{i}-b_{i})=0,\quad\eta^{*}_{i}(-a_{i}+ x_{i}^{*})=0,\quad i=1,2,\dots,n.
 \end{equation*}
It implies that the conclusion holds.
\end{proof}
\section{SSAL algorithm}
\label{sec:Alternating direction method}
In this section, we propose a SSAL method for solving problem (P). Firstly, we derive the augmented Lagrangian decomposition formulation ($P_{2}$) of $(P)$. Secondly, we discuss two subproblems obtained from fixing some variables respectively. Thirdly, we propose the method for solving problem (P) which minimizes (P$_{2}$) with respect to $x$, $y$ in an alternating fashion while updating $\lambda$ in the iteration. Finally, we establish some convergence results for problem (P).

Introducing $y=x$, we can rewrite (P) as
\begin{eqnarray*}
{\rm(P_{1})}~~~~\min~~&&f(x) \\
 {\rm s.t.}~ ~~&&x\in X,\\
&&y=x,\\
&& \|y\|_{0}\leq K,\\
&&y_{i}\in\{0\}\cup[a_{i},b_{i}],\quad  i=1,2,\dots,n.
\end{eqnarray*}
%
Then, we define the augmented Lagrangian function for (P$_{1}$) as follows:
\begin{equation*}
L_{\rho}(x,y,\lambda)=f(x)+\lambda^{T}(y-x)+\frac{\rho}{2}\|y-x\|^{2}
\end{equation*}
where $\lambda\in R^{n}$ is the Lagrangian multiplier associated with the equality constraint $y=x$ and $\rho>0$ is a penalty parameter. The resulting augmented Lagrangian decomposition formulation is:
\begin{eqnarray*}
{\rm(P_{2})}~~~~\min~~&&L_{\rho}(x,y,\lambda) \\
 {\rm s.t.}~ ~~ &&x\in X,\\
&& y\in Y,
\end{eqnarray*}
where
\begin{equation}\label{S}
 Y=\{y\in R^{n}|~\|y\|_{0}\leq K, y_{i}\in\{0\}\cup [a_{i},b_{i}]\}.
\end{equation}
 In the following part, we first discuss the subproblem $P_{x}$ obtained from fixing the variable $y$, and then we discuss the subproblem $P_{y}$ for fixing the variable $x$.
\subsection{The convex programming problem of variables x}
\label{subsec:The convex programming problem of variables x}
We notice that for given $\lambda=\bar{\lambda}$, when $y=\bar{y}\in Y$ is fixed, (P$_{2}$) becomes:
\begin{eqnarray*}
{\rm(P_{x})}~~~~~\min  ~~&&L_{\rho}(x,\bar{y},\bar{\lambda})\\
{\rm s.t.}~ ~~ && x\in X.
\end{eqnarray*}
Note that $X$ is a convex set and the objective function of (P$_{x}$) is a convex function of $x$. Thus, (P$_{x}$) is a convex programming problem of variables $x$.
\subsection{Minimization of separable quadratic function}
\label{subsec:Minimization of separable quadratic function}
We consider the subproblem of (P$_{2}$) for given $\lambda=\bar{\lambda}$ when $x=\bar{x}\in X$ is fixed, problem (P$_{2}$) becomes
 \begin{eqnarray*}
{\rm(P_{y})}~~~~\min  ~~&&L_{\rho}(\bar{x},y,\bar{\lambda})\\
{\rm s.t.}~ ~~ && y\in Y.
\end{eqnarray*}
%
%
The object function $L(\bar{x},y,\bar{\lambda})$ is a quadratic function of variable $y$, so we can simplify the objective function $L(\bar{x},y,\bar{\lambda})$, and then transform the subproblem $(P_{y})$ to the $2-$norm approximation problem with semi-continuous variables and the cardinality constraint.
\begin{equation*}
\begin{split}
L_{\rho}(\bar{x},y,\bar{\lambda})&=f(\bar{x})+\bar{\lambda}^{T}(y-\bar{x})+\frac{\rho}{2}\|y-\bar{x}\|^{2}\\
&=f(\bar{x})-\bar{\lambda}^{T}\bar{x}+\frac{\rho}{2}\|\bar{x}\|^{2}+\frac{\rho}{2}[\|y\|^{2}-2(\bar{x}-\frac{1}{\rho}\bar{\lambda})^{T}y].
\end{split}
\end{equation*}
Let $w=\bar{x}-\frac{1}{\rho}\bar{\lambda}$, the object function  $L(\bar{x},y,\bar{\lambda})$ can be rewritten as:
\begin{equation}\label{abouty}
\begin{split}
L_{\rho}(\bar{x},y,\bar{\lambda})&=f(\bar{x})-\bar{\lambda}^{T}\bar{x}+\frac{\rho}{2}\|\bar{x}\|^{2}+\frac{\rho}{2}(\|y\|^{2}-2w^{T}y)~\\
&=f(\bar{x})-\bar{\lambda}^{T}\bar{x}+\frac{\rho}{2}\|\bar{x}\|^{2}-\frac{\rho}{2}\|w\|^{2}+\frac{\rho}{2}\|y-w\|^{2}\\
&=f(\bar{x})-\frac{1}{2\rho}\|\bar{\lambda}\|^{2}+\frac{\rho}{2}\|y-w\|^{2}
\end{split}
\end{equation}
For given $\bar{\lambda}$, when $x=\bar{x}\in X$ is fixed, $\frac{1}{2\rho}\|\bar{\lambda}\|^{2}$, $w=\bar{x}-\frac{1}{\rho}\bar{\lambda}$ are constants.
From (\ref{abouty}), the minimization of the problem $(P_{y})$ is equivalent to the following problem :
\begin{eqnarray*}
{\rm(P_{3})}~~~~\min~~&&\sum_{i=1}^{n}(y_{i}-w_{i})^{2}\\
 {\rm s.t.}~ ~~&&y\in Y.
\end{eqnarray*}
Introducing a variable $z_{i}\in \{0,1\}$ to indicate the zero or nonzero status of each decision variable $y_{i}$, $i=1,\dots n$. Problem ($P_{3}$) can be reformulated to the following mixed integer program (MIP), which can be solved by CPLEX 12.6.
\begin{eqnarray*}
{\rm(MIP)}~~~~\min  ~~&&\sum_{i=1}^{n}(y_{i}-w_{i})^{2}\\
 {\rm s.t.}~ ~~&&e^{T}z\leq K,\quad z_{i}a_{i}\leq y_{i}\leq z_{i}b_{i},\\
 &&z_{i}\in \{0,1\},\quad i=1,2,\dots,n.
\end{eqnarray*}
Let $\bar{y}$ be the optimal solution of the problem ($P_{y}$), then there exists $\bar{z}$ together with $\bar{y}$ such that $(\bar{y},\bar{z})$ is the optimal solution of the problem ($P_{3}$).
Because the CPLEX 12.6 ignores the feature of this problem (MIP), it is slow, especially for large-scale problems. So we propose new method to obtain the optimal solution $(\bar{y},\bar{z})$ of (MIP).

Note that the value of variable $z_{i}$ only has two situations: either $z_{i}=0$ or  $z_{i}=1$. Moreover $z_{i}=0$ indicates $y_{i}=0$ and $z_{i}=1$ indicates $a_{i}\leq y_{i}\leq b_{i}$, $i=1,2,\dots,n$. We can get the optimal solution $(\bar{y},\bar{z})$ of problem (MIP) successively.

Firstly, taking into account that either $z_{i}=1$ or  $z_{i}=0$,
we introduce auxiliary variables $q_{i}$ and $r_{i}$ as follows:
\begin{equation*}
 q_{i}=\min_{ z_{i}=0}(y_{i}-w_{i})^{2}=\min_{y_{i}=0}(y_{i}-w_{i})^{2}=w_{i}^{2} , \mbox{\ where\ }\bar{y}_{i}=0,
\end{equation*}
\begin{eqnarray*}
 r_{i}&&=\min_{ z_{i}=1}(y_{i}-w_{i})^{2}=\min_{a_{i}\leq y_{i}\leq b_{i}}(y_{i}-w_{i})^{2}\\
&&=
\begin{cases}
\begin{split}
&(a_{i}-w_{i})^{2}    &{\rm{if}}\quad w_{i}<a_{i},~ & {\rm{where}}~\bar{y}_{i}=a_{i},\\
&0         &{\rm{if}} \quad a_{i}\leq w_{i}\leq b_{i},~ & {\rm{where}} ~\bar{y}_{i}=w_{i},\\
&(b_{i}-w_{i})^{2},    &{\rm{if}}\quad w_{i}>b_{i},~ & {\rm{where}} ~\bar{y}_{i}=b_{i}.
\end{split}
 \end{cases}
 \end{eqnarray*}
%

Secondly, we transform the problem (MIP) to the following problem, which is a optimization problem over $z$ :
\begin{equation}\label{11}
\begin{split}
&\min\{\sum_{i=1}^{n}(y_{i}-w_{i})^{2}~|e^{T}z\leq K,z\in \{0,1\}^{n},z_{i}a_{i}\leq y_{i}\leq z_{i}b_{i} ,i=1,2,\dots,n\}\\
&=\min\{\sum_{i=1}^{n}r_{i}z_{i}+(1-z_{i})q_{i}~|~e^{T}z\leq K,z\in \{0,1\}^{n}\}\\
&=\min\{\sum_{i=1}^{n}q_{i}+(r_{i}-q_{i})z_{i}~|~e^{T}z\leq K,z\in \{0,1\}^{n}\}\\
&=\sum_{i=1}^{n}w_{i}^{2}+\min\{\sum_{i=1}^{n}(r_{i}-w_{i}^{2})z_{i}~|~e^{T}z\leq K,z\in \{0,1\}^{n}\}
\end{split}
\end{equation}
From (\ref{11}), the optimal solution of problem (MIP) is obtained by solving
  \begin{equation*}
 {\rm(ILP)}~~~ \min\{\sum_{i=1}^{n}(r_{i}-w_{i}^{2})z_{i}~|~e^{T}z\leq K,z\in \{0,1\}^{n}\}.
  \end{equation*}
If $K=n$, the cardinality constraint is void, hence we are left with the simple problem
\begin{equation*}
 \min\{\sum_{i=1}^{n}(r_{i}-w_{i}^{2})z_{i}~|z\in \{0,1\}^{n}\},
 \end{equation*}
whose optimal solution is given by:
\begin{equation*}
\bar{z}_{i}=
\begin{cases}
1,& {\rm{if}}~ r_{i}-w_{i}^{2}\leq 0,\\
0,&{\rm{otherwise}}.
 \end{cases}
\end{equation*}
In order to deal with the general situation (i.e., if $K<n$) we define $v=r-w^{2}$.
Without loss of generality we assume  that
$v$ is sorted in ascending order: $v_{1}\leq v_{2}\leq \dots \leq v_{n}$. Then
we have the following theorem.
\begin{theorem}\label{qiuz}
The optimal solution $\bar{z}$ of problem (ILP) is given by
\begin{equation*}
\bar{z}_{i}=
\begin{cases}
1,&{\rm{if}}~i\leq K  ~{\rm{and}}~v_{i}\leq 0 ,\\
0,&{\rm{otherwise}}.
 \end{cases}
\end{equation*}
  \end{theorem}
\begin{proof}
It is obvious that the optimal value is either negative or zero.
If the number of the negative entries of $v$ is less than $K$,  the optimal value is the sum of all negative entries of $v$, so the optimal solution is:
\begin{equation*}
\bar{z}_{i}=
\begin{cases}
1,&{\rm{if}}~v_{i}\leq 0,\\
0,&{\rm{otherwise}}.
 \end{cases}
\end{equation*}
If the number of the negative entries of $v$ is more than $K$, the optimal value is the sum of the $k$ least entries of $v$,
so the optimal solution becomes.
\begin{equation*}
\bar{z}_{i}=
\begin{cases}
1,&{\rm{if}}~i\leq K,\\
0,&{\rm{otherwise}}.
 \end{cases}
\end{equation*}
\end{proof}

After Theorem \ref{qiuz}, the explicit solution of $(P_{y})$ can be given as in
the next theorem.

 \begin{theorem}\label{updateyz}
For given $\lambda=\bar{\lambda}$ and $x=\bar{x}\in X$, the optimal solution $\bar{y}$ of the problem ($P_{y}$) is given by
\begin{equation*}
\bar{y}_{i}=\bar{z}_{i}\cdot\min\{b_{i},\max\{w_{i},~a_{i}\}\},\quad for\quad  i=1,2,\dots n,
\end{equation*}
where $\bar{z}$ can be obtained from Theorem \ref{qiuz}.
 \end{theorem}


Yet we are ready to describe our algorithm, named Splitting and Successively solving Augmented Lagrangian method or SSAL method, as in Algorithm \ref{SSAL}.
\label{subsec:Algorithm}

\begin{algorithm}[!htb]
\DontPrintSemicolon 
\KwIn{Choose tolerance parameter $\epsilon \geq0$,
multiplier vector $\lambda^{0}$, penalty parameter $\rho>0$ and the step-size $\omega$.
Set $y^{0}\in Y$ and the iteration counter $k=0$.}
\KwOut{An approximate optimal solution $(x, z)$  of problem ($P$).}
Solve the subproblem problem (P$_{x}$) with $\bar{y}=y^{0}$ to obtain an optimal solution $x^{0}$\;
\While{$\|x^{k}-y^{k}\|^{2}_{2}> \epsilon$ }{
    update $z^{k+1}$ and $y^{k+1}={\rm{argmin}}_{y\in Y}L_{\rho}(x^{k},y,\lambda^{k})$, according to Theorem~\ref{qiuz} and Theorem~\ref{updateyz} \;
   update $x^{k+1}={\rm{argmin}}_{x\in X}L_{\rho}(x,y^{k+1},\lambda^{k})$\;
  update the multipliers by
  \begin{equation}\label{updatelambda}
  \lambda_{i}^{k+1}=\lambda_{i}^{k}+\omega\rho[y_{i}^{k+1}-x_{i}^{k+1}],\quad i=1,2,\dots,n
  \end{equation}\;
    increase $k$ by one and continue\;
}
\Return{$x=x^{k}$ and $z=z^{k}$}\;
\caption{SSAL method for solving the problem (P)\label{SSAL}}
\end{algorithm}

\subsection{Convergence analysis}
\label{subsec:Convergence analysis}
In this subsection, we show that any accumulation point of the sequence generated by SSAL method satisfies the first order stationary conditions.
%
\begin{theorem}\label{conver}
Let $(\hat{x},\hat{y},\hat{\lambda})$ be any accumulation point of $(x^{k},y^{k},\lambda^{k})$ generated by algorithm 1, $J=\{1\leq j\leq n:\hat{x}_{j} \neq 0\}$ and $r=|J|$. Assuming that
 $\{\lambda^{k}\}$ is bounded, we have
 \begin{equation}\label{iteration}
 \lim_{k\rightarrow\infty}[(x^{k+1},y^{k+1},\lambda^{k+1})-(x^{k},y^{k},\lambda^{k})]=0
 \end{equation}
  and the Robinson condition $(\ref{Robinson})$ holds at $\hat{x}$  for such $J$ .
Then, $\hat{x}$ satisfies the first order stationary conditions $(\ref{opt})$.
\end{theorem}

The assumption on (\ref{iteration}) and the boundedness of the multiplier vectors in Theorem \ref{conver} are the standard condition in the convergence analysis of the augmented Lagrangian methods for nonconvex optimization problems (see \cite{Luo07}). Similar conditions have been used in the convergence analysis of alternating direction methods of multipliers for nonconvex optimization problems (see \cite{Shen14} and \cite{bai13}).

Before proceeding to the proof of Theorem \ref{conver}, we prove two lemmas.
\begin{lemma}\label{kth}
Let $J^{k+1}=\{1\leq j\leq n:y^{k+1}_{j} \neq 0\}$ and $|J^{k+1}|=r ^{k+1}$, then there exists
$(\nu^{k},\eta^{k},\kappa^{k})\in R^{n}\times R^{n}\times R^{n}$ together with the $x^{k}$, $x^{k+1}$, $y^{k+1}$ and penalty parameter $\rho$  satisfying:
\begin{eqnarray}
 0\in \nabla
 f(x^{k+1})+\nu^{k}-\eta^{k}+\kappa^{k}+\rho(x^{k+1}-x^{k})+N_{X}(x^{k+1}),
 \label{kopt1}  \\
  \nu_{i}^{k}\geq 0,\eta_{i}^{k}\geq 0,\quad \nu^{k}_{i}(
  y^{k+1}_{i}-b_{i})=0,\eta^{k}_{i}( y_{i}^{k+1}-a_{i})=0,\quad i=1,2,\dots,n,
  \label{kopt2}\\
  \kappa_{j}^{k}=0,\quad j\in J^{k+1} . \label{kopt3}
 \end{eqnarray}

\end{lemma}
\begin{proof}
Using (\ref{11}) and the same augment in Theorem \ref{necessary optimality conditions},  we have $y^{k+1}$ is also a local minimizer of the following problem:
\begin{eqnarray*}
 {\rm(P_{4})}~~~\min~~ &&\|y-x^{k}+\frac{1}{\rho}\lambda^{k}\|^{2}\\
  {\rm s.t.} ~~~&&I^{T}_{J^{k+1}}(y-b)\leq 0\\
 &&I^{T}_{J^{k+1}}(-y+a)\leq 0\\
  &&I^{T}_{\bar{J}^{k+1}}y= 0.
\end{eqnarray*}
%

Now, this problem is a easy to solve convex optimization problem. And there exists $(\nu^{k},\eta^{k},\kappa^{k})$ together with $y^{k+1}$ satisfying:
   \begin{eqnarray}
   0\in \rho(y^{k+1}-x^{k}+\frac{1}{\rho}\lambda^{k})+\nu^{k}-\eta^{k}+\kappa^{k},\label{Roby1}\\
     \nu_{i}^{k}\geq 0,\eta_{i}^{k}\geq 0,\quad \nu^{k}_{i}(y^{k+1}_{i}-b_{i})=0,\quad \eta^{k}_{i}( y_{i}^{k+1}-a_{i})=0,i=1,\dots,n,\label{Roby2}\\
       \kappa_{j}^{k}=0,\quad j\in J^{k}    \label{Roby3}
\end{eqnarray}

Using the definition of $x^{k+1}$ in algorithm 1 and  the Theorem 3.34 of \cite{A.Ruszcunski}
, we get the conclusion that $x^{k+1}$ satisfies the following relationship:
\begin{equation}\label{Robinxk}
\begin{split}
0\in \nabla f(x^{k+1})-\lambda^{k}-\rho(y^{k+1}-x^{k+1})+N_{X}(x^{k+1}).
\end{split}
\end{equation}
Combining $(\ref{Roby1})$ and $(\ref{Robinxk})$, we obtain
\begin{equation*}
\begin{split}
0&\in \rho(y^{k+1}-x^{k}+\frac{1}{\rho}\lambda^{k})+\nu^{k}-\eta^{k}+\kappa^{k}\\
&+\nabla f(x^{k+1})-\lambda^{k}-\rho(y^{k+1}-x^{k+1})+N_{X}(x^{k+1}),
\end{split}
\end{equation*}
which implies that
\begin{equation*}
0\in \nabla f(x^{k+1})+\nu^{k}-\eta^{k}+\kappa^{k}+\rho(x^{k+1}-x^{k})+N_{X}(x^{k+1}).
\end{equation*}
Combining this observation and $(\ref{Roby2})$, $(\ref{Roby3})$ , we see that the conclusion holds.
\end{proof}

The proof below uses that $a_{i}\neq 0$, $b_{i}\neq 0$. In the compressed sensing problem, we can solving this issue by taking $a_{i}>0$, but very small. In practice, this works.

\begin{lemma}
There exists a subsequence $\hat{K}$  such that $\{(\nu^{k},\eta^{k},\kappa^{k})\}_{k\in\hat{ K}}$ is bounded.
\end{lemma}
\begin{proof}
Since $(\hat{x},\hat{y},\hat{\lambda})$ is an accumulation point of $(x^{k},y^{k},\lambda^{k})$, there exists a subsequence $\{(x^{k},y^{k},\lambda^{k})\}_{k\in\hat{ K}}\rightarrow (\hat{x},\hat{y},\hat{\lambda})$.
It follows from (\ref{updatelambda}) that $y_{i}^{k+1}-x_{i}^{k+1}=(\lambda_{i}^{k+1}-\lambda_{i}^{k})/\omega \rho$. Taking into account the assumption (\ref{iteration}),
  we have
  \begin{equation}\label{xy}
  \hat{x}=\hat{y}.
  \end{equation}

Suppose for contradiction that it is unbounded.
By passing to a subsequence if necessary, we can assume that $\|(\nu^{k},\eta^{k},\kappa^{k})\|\rightarrow\infty$.
Let
\begin{equation*}
(\bar{\nu}^{k},\bar{\eta}^{k},\bar{\kappa}^{k})
=(\nu^{k},\eta^{k},\kappa^{k})/
\|(\nu^{k},\eta^{k},\kappa^{k})\|,
\end{equation*}
then $\|(\bar{\nu}^{k},\bar{\eta}^{k},\bar{\kappa}^{k})\|=1$.
There exists a convergent subsequence $K\subseteq\hat{K}$, such that
\begin{equation*}
(\bar{\nu}^{k},\bar{\eta}^{k},\bar{\kappa}^{k})\rightarrow
(\bar{\nu},\bar{\eta},\bar{\kappa}), \quad as \quad k\in K\rightarrow \infty ,\quad K\subseteq\hat{K},
\end{equation*}
Clearly,$\ |(\bar{\nu},\bar{\eta},\bar{\kappa})\|=1$, $\bar{\nu}\geq 0$ and $\bar{\eta}\geq 0$ .
%
Divide both sides by $\|(\nu^{k},\eta^{k},\kappa^{k})\|$ in (\ref{kopt1}) and take limits as $k\in K\rightarrow \infty$.
Noticing $f(x)$ is a continuously differentiable convex function and using the the semi continuity of $N_{X}(\cdot)$ (Lemma 2.42 of \cite{A.Ruszcunski}), we obtain
\begin{equation}\label{normal}
-\bar{\nu}+\bar{\eta}-\bar{\kappa}\in N_{X}(\hat{x})
\end{equation}

%
Dividing both sides by $\|(\nu^{k},\eta^{k},\kappa^{k})\|$ in ($\ref{kopt2}$) and ($\ref{kopt3}$), we have
\begin{eqnarray}
\bar{\nu}^{k}_{i}( y^{k+1}_{i}-b_{i})=0,\quad\bar{\eta}^{k}_{i}(y_{i}^{k+1}-a_{i})=0,\quad i=1,2\dots,n, \label{bark1}\\
\bar{\kappa}_{j}^{k}=0,\quad j\in J^{k+1}. \label{bark2}
 \end{eqnarray}

Taking limits as $k\in K\rightarrow \infty$ in (\ref{bark1}), using (\ref{xy}), we have
\begin{equation}\label{etanu}
\bar{\nu}_{i}( \hat{x}_{i}-b_{i})=0,\quad \bar{\eta}_{i}( \hat{x}_{i}-a_{i})=0,\quad i=1,2,\dots,n.
\end{equation}
By the definition of $J$, we have $\hat{x}_{i}=\hat{y}_{i}=0$, for $i\in\bar{J}$. Combining this with $(\ref{etanu})$, $a_{i}\neq0$ and $b_{i}\neq0$, we obtain
$\bar{\nu}_{i}=0$, $\bar{\eta}_{i}=0$. Moreover, we have $\|\bar{\nu}\|=\|\bar{\nu}_{J}\|$,
$\|\bar{\eta}\|=\|\bar{\eta}_{J}\|$,
 $\bar{\nu}=I_{J}\bar{\nu}_{J}$ and
 $\bar{\eta}=I_{J}\bar{\eta}_{J}$.

Now we show that for $l\in J$, $\bar{\kappa}_{l}=0$.
For $l\in J$, $\hat{y}_{l}\neq0$. Note that $y_{l}^{k}\rightarrow
 \hat{y}_{l}$, there exists a sufficiently large $\hat{k}$, such that
 $y_{l}^{k}\neq0$ for $k>\hat{k}$.
Combining this with (\ref{bark2}), we have $\bar{\kappa}_{l}^{k}=0$ for $l\in J$, $k>\hat{k}$. Taking limits as $k\in K\rightarrow \infty$, we obtain  $\bar{\kappa}_{l}=0$, for $l\in J$.
  Thus, we have $\|\bar{\kappa}\|=\|\bar{\kappa}_{\bar{ J}}\|$ and $\bar{\kappa}=I_{\bar{ J}}\bar{\kappa}_{\bar{ J}}$.

Since Robinson's condition $(\ref{Robinson})$ is satisfied at $\hat{x}$,
 there exist
 \begin{equation}\label{d}
 \tilde{d}\in T_{X}(\hat{x}),\quad \tilde{s}\in R^{r},\quad \tilde{t}\in R^{r},
  \end{equation}
such that
 \begin{equation}\label{constitution}
 \tilde{s}_{A(\hat{x})}\leq0,\quad
 \tilde{t}_{B(\hat{x})}\leq0,\quad
 I^{T}_{J}\tilde{d}-\tilde{s}=-\bar{\nu}_{J},\quad-I^{T}_{J}\tilde{d}-\tilde{t}=-\bar{\eta}_{J}, \quad I^{T}_{\bar{J}}\tilde{d}=-\bar{\kappa}_{\bar{J}} .
\end{equation}

Using (\ref{normal}), (\ref{d}), (\ref{constitution}), $\|\bar{\nu}\|=\|\bar{\nu}_{J}\|$,
$\|\bar{\eta}\|=\|\bar{\eta}_{ J}\|$,
$\|\bar{\kappa}\|=\|\bar{\kappa}_{\bar{J}}\|$,
$\bar{\nu}=I_{ J}\bar{\nu}_{J}$,
$\bar{\eta}=I_{ J}\bar{\eta}_{ J}$ and $\bar{\kappa}=I_{\bar{J}}\bar{\kappa}_{\bar{J}}$ we have
 \begin{equation*}
\begin{split}
\|\bar{\nu}\|^{2}+\|\bar{\eta}\|^{2}+\|\bar{\kappa}\|^{2}
&=\|\bar{\nu}_{J}\|^{2}+\|\bar{\eta}_{J}\|^{2}+\|\bar{\kappa}_{\bar{J}}\|^{2}\\
&=-[(-\bar{\nu}_{J})^{T}\bar{\nu}_{J}+(-\bar{\eta}_{J})^{T}\bar{\eta}_{J}+(-\bar{\kappa})_{\bar{J}}^{T}\bar{\kappa}_{\bar{J}}]\\
&=-[(I^{T}_{J}\tilde{d}-\tilde{s})^{T}\bar{\nu}_{J}+(-I^{T}_{J}\tilde{d}-\tilde{t})^{T}\bar{\eta}_{J}+(I^{T}_{\bar{J}}\tilde{d})^{T}\bar{\kappa}_{\bar{J}}]\\
&=\tilde{d}^{T}(-I_{J}\bar{\nu}_{J}+I_{J}\bar{\eta}_{J}-I_{\bar{J}}\bar{\kappa}_{\bar{J}})+(s^{T}\bar{\eta}_{J}+t^{T}\bar{\nu}_{J})\\
&=\tilde{d}^{T}(-\bar{\nu}+\bar{\eta}-\bar{\kappa})+(\tilde{s}^{T}\bar{\nu}+\tilde{t}^{T}\bar{\eta})=\tilde{s}^{T}\bar{\nu}+\tilde{t}^{T}\bar{\eta}
\end{split}
\end{equation*}
In addition, it is following from (\ref{etanu}) that if $\hat{x}_{i}\neq b_{i}$, then $\bar{\nu}_{i} =0$. Thus $\tilde{s}_{i}\bar{\nu}_{i}=0$. If $\hat{x}_{i}= b_{i}$, then $\tilde{s}_{i}\leq0$. combing this $\bar{\nu}\geq 0$, we have $\tilde{s}_{i}\bar{\nu}_{i}\leq0$.  Hence, $\tilde{s}^{T}\bar{\nu}\leq 0$.
By a similar argument as above, one can show that $\tilde{t}^{T}\bar{\eta}\leq 0$.

Using these relations, we have
$$\|\bar{\nu}\|^{2}+\|\bar{\eta}\|^{2}+\|\bar{\kappa}\|^{2}
=\tilde{s}^{T}\bar{\nu}+\tilde{t}^{T}\bar{\eta}
\leq 0.$$
It yields $\|(\bar{\nu},\bar{\eta},\bar{\kappa})\|=0$, which contradicts $\|(\bar{\nu},\bar{\eta},\bar{\kappa})\|=1$. Therefore, the subsequence $\{(\nu^{k},\eta^{k},\kappa^{k})\}_{k\in \hat{K}}$ is bounded.
\end{proof}

\textbf{Proof of Theorem \ref{conver}}
\begin{proof}
Taking into account that $\{(\nu^{k},\eta^{k},\kappa^{k})\}_{k\in \hat{K}}$ is bounded,
then there exists a subsequence $\bar{K}\subseteq K$ such that $(\nu^{k},\eta^{k},\kappa^{k})\rightarrow(\hat{\nu},\hat{\eta},\hat{\kappa})$
as $k\in \bar{K}\rightarrow \infty$.
Taking limits in $(\ref{kopt1})$,$(\ref{kopt2})$ and $(\ref{kopt3})$ as $k\in \bar{K}\rightarrow \infty$, using the relation $\hat{x}=\hat{y}$, the assumption (\ref{iteration}) and the semicontinuity of $N_{X_{0}}(\cdot)$, we get
\begin{equation*}
\begin{split}
0\in \nabla f(\hat{x})+\hat{\nu}-\hat{\eta}+\hat{\kappa}+N_{X}(\hat{x})\\
\hat{\nu}_{i}\geq 0, \quad \hat{\eta}_{i} \geq 0,\quad \hat{\nu}_{i} (\hat{x}_{i}-b_{i})=0,\quad \hat{\eta}_{i}(-a_{i}+ \hat{x}_{i})=0, \quad i=1,2,\dots,n,\\
\hat{\kappa}_{i}=0,\quad i\in J
\end{split}
\end{equation*}
Hence, $(\hat{\nu},\hat{\eta},\hat{\kappa})$ together with $\hat{x}$ satisfies the
first order stationary conditions $(\ref{opt})$.

\end{proof}

\section{Numerical results}
\label{sec:Numerical results}
In this section, we conduct numerical experiments to test the performance of SSAL proposed in Section $\ref{sec:Alternating direction method}$. We select the portfolio selection problem ($P_{NS}$) \cite{cui} and the compressed sensing problem (NCSB) as examples.
We use real market data to construct the test problems (P$_{NS}$).
We also present simulation study of the performance of SSAL for the portfolio selection problem ($P_{NS}$) \cite{cui} and the compressed sensing problem (NCSB).
The experiments demonstrate that as far as the quality of approximate solutions is concerned, SSAL outperforms the PD method and is almost as good as the CPLEX and the MIQCQP$_{1}$.
At the same time, the computational time for SSAL is significantly smaller than the other three methods.
The numerical tests are implemented in MATLAB 7.14 and run on a PC (2.10G Hz, 2GB RAM).
\subsection{Portfolio selection problem with nonsystematic risk constraint}
\label{subsec1:Portfolio}
In this subsection, we verify the efficiency and stability of SSAL method in portfolio selection problem (P$_{NS}$) of \cite{cui}:
\begin{eqnarray*}
{\rm(P_{NS})}~~~~\min  ~~ &&f(x)=x^{T}(Q+D)x \\
 {\rm s.t.}~ ~~ &&g(x)=x^{T}Dx\leq \sigma_{0},\\
&& \alpha^{T}x\geq\rho_{0},\quad e^{T}x=1,\\
&&  \|x\|_{0}\leq K ,  \\
 &&  x_{i}\in \{0\}\cup[a_{i},b_{i}],\quad i\in \{1,2,\dots,n\},
\end{eqnarray*}
where $Q$ is a semi-definition matrix and $D$ is a diagonal matrix.
Clearly, problem (P$_{NS}$) is in the form of (P), thus SSAL proposed in Section \ref{sec:Alternating direction method} can be suitably applied to solve (P$_{NS}$). The $x-$update step involves a quadratically constrained quadratic program (QCQP), which is
\begin{eqnarray*}
{\rm(QCQP)}~~~~\min  ~~ &&x^{T}\left(Q+D+\frac{\rho}{2}E\right)x-\left(\bar{\lambda}+\rho\bar{y}\right)^{T}x\\
 {\rm s.t.}~ ~~ &&g(x)=x^{T}Dx\leq \sigma_{0},\quad \alpha^{T}x\geq\rho_{0},\\
&&e^{T}x=1,\quad 0\leq x\leq b.
\end{eqnarray*}
We use the  CPLEXQCP solver in CPLEX 12.6 \cite{cplex13} with MATLAB interface to solve the (QCQP) problem.
The $y-$update step has exactly the same expression as in the Theorem \ref{updateyz}.
It is obviously that the main computational effort of SSAL lies in solving the subproblem (QCQP).
In our test, the parameters in SSAL method are set as follows: the step-size $\omega=0.3$, the initial Lagrangian multiplier $\lambda^{0}=0$ and the initial penalty $\rho=1$. We set the prescribed weekly return level $\rho_{0}=2\times10^{-3}$, the prescribed nonsystematic risk level $\sigma_{0}=1\times10^{-3}$,  $a_{i}=0.01$, and $b_{i}=0.3$, $i=1,2,\dots,n$. The cardinality upper bound $K=10$, and the tolerance parameter is set as $\epsilon=10^{-4}$.
\subsubsection{Real data set}
\label{subsec:Real data set}
In this subsection, we use the weekly return data of 481 stocks from Standard\&Poor's 500 index between 2005 and 2010. We take linear regressions on 10 sector indexes of Standard\&Poor's 500 index to construct a sector factor model.
The objective function of (P$_{NS}$) is the covariance obtained by the sector factor model while the quadratic constraint controls the nonsystematic risk $x^{T}Dx$ in the sector factor model.
We generate 5 instances for each test problem with the same size (n=50,100,150). For each instance, we choose $n$ stocks from the 481 stocks randomly.

In order to test the accuracy of solution obtained by SSAL method, we compare SSAL with the CPLEXMIQCP solver in CPLEX 12.6 \cite{cplex13}.
Introducing a variable $z_{i}\in \{0,1\}$ to indicate the zero or nonzero status of each decision variable $x_{i}$, $i=1,2,\dots,n$,
problem ($P_{NS}$) can be reformulated as follows:
\begin{eqnarray*}
{\rm(MIP_{0})}~~~~~\min~~&&x^{T}(Q+D)x \\
 {\rm s.t.}~~ ~&&x^{T}Dx\leq \sigma_{0},\\
 && \alpha^{T}x\geq\rho_{0},\quad e^{T}x=1,\\
&& e^{T}z\leq K, \quad z_{i}a_{i}\leq x_{i}\leq z_{i}b_{i},\\
&& z_{i}\in \{0,1\},\quad i=1,2,\dots,n.
\end{eqnarray*}
The experiments are conducted by using the default setting of CPLEX 12.6 with the maximum CPU time set equal to 3600 s.

To illustrate the performance of SSAL for test problems (P$_{NS}$) specifically,
we present three comparison results in terms of the computational time in seconds, the number of iteration and the objective value explored by SSAL method and CPLEX 12.6 in Table 1. We also record the relative difference (rel.dif) in Table 1, where rel.dif:=$\left(fval(x_{SSAL}\right)-fval\left(x_{CPLEX})\right)/fval(x_{CPLEX})$.
\begin{table}
\caption{SSAL versus CPELX 12.6 on the (P$_{NS})$ problem}
        \center{\begin{tabular}{ccccccccc}
        \hline
    \multirow{2}{*}{n}&\multicolumn{3}{c}{SSAL}&&\multicolumn{2}{c}{CPLEX}\\
    \cline{2-4}   \cline{6-7}\\
& fval$\times10^{-4} $& cputime & iter &&fval$\times10^{-4}$&cputime&rel.dif$\times 10^{-2}$ \\
    \hline
50 &   2.3503 &   \textcolor{red}{0.3746}&  7&&    2.2981  & 0.9093&  2.2725
\\
50 &   2.7138&   \textcolor{red}{ 0.1697}&  5 &&     2.6584  &1.4307&2.0845  \\
50 &    3.7464&   \textcolor{red}{   0.2343}&  5& &    3.6557  &    2.0024 &2.4788
 \\
50 &     3.5357 &  \textcolor{red}{ 0.6295}&  16 &&   3.5446  &  2.7720&   -0.0507 \\
50 &  2.1145&  \textcolor{red}{  0.2257}& 7 & &   2.0819  &  1.3988&    1.5695
 \\
   \hline
     100&    2.3550 &   \textcolor{red}{   0.7521} & 8& &   2.3294 &  9.8475 & 1.0969\\
          100&    1.4998&    \textcolor{red}{   0.7698} &10&& 1.5232&    15.8758 & -1.5421 \\
    100&   4.3390 &\textcolor{red}{   0.9155} &  10 & & 4.3077 &      6.5231 &0.7267
\\
     100&      2.0593 &   \textcolor{red}{   1.0874}  & 14& &  2.0380  &     46.6908& 1.0468
 \\
     100&    1.6443&   \textcolor{red}{   1.3415}  & 15 &&    1.6683&   84.1186 & -1.4396\\
         \hline
150 &      1.4061&  \textcolor{red}{  2.1155}  &  11& &1.3803 &   168.8789&1.8750
\\
 150 &     2.8602 &  \textcolor{red}{   1.4629} &  8 &&     2.8579  &  34.4529&0.0799
\\
 150 &    1.8396  &   \textcolor{red}{   2.1173}  &   13&& 1.8023 &    312.4525&  2.0674
\\
150 &       1.4054&  \textcolor{red}{    3.7521} & 16&&    1.3631 &   497.5107& 3.1043\\
 150 &   1.2593 &   \textcolor{red}{  3.9929}  & 18&&     1.2492 & 732.9191& 0.8100
\\
    \hline
 \end{tabular}}
 \label{table1}
 \end{table}

The numerical results in Table \ref{table1} show that the cost of computational efforts of SSAL is significantly less than the well-known
solver CPLEX 12.6 for solving (P$_{NS}$) on a real world data. For instance, the example listed the last row in the Table 1 shows that CPU time of SSAL is almost 200 times faster than that of the CPLEX. Obviously, we conclude that the SSAL method outperforms the solver CPLEX 12.6. Note that the optimal objective value of the problem (P$_{NS}$) solved by SSAL is almost same to that of the CPLEX. The largest absolute value of the relative difference (“rel.dif”) is less $4\%$.

Then we continue to compare SSAL method with the methods provided by Cui et al in \cite{cui}, where a new model defined as MIQCQP$_{1}$ as follows:
\begin{eqnarray*}
{\rm(MIQCQP_{1})}~~~~\min  ~~ &f(x,\delta)=x^{T}(Q+D)x+\sum_{i-1}^{n}d_{i}\delta_{i}\\
{\rm s.t.}~ ~~ &g(x)=\sum_{i=1}^{n}d_{i}\delta_{i}\leq \sigma_{0},\quad \alpha^{T}x\geq\rho_{0}, \\
&  e^{T}x=1, \quad\|x\|_{0}\leq K,\quad x_{i}^{2}\leq \delta_{i}z_{i},\\
&  x_{i}\in \{0\}\cup[a_{i},b_{i}],\quad i\in \{1,2,\dots,n\}.
\end{eqnarray*}
Problem (MIQCQP$_{1}$) is also solved by the mixed integer QCP solver
in CPLEX 12.6 with MATLAB interface using continuous relaxation for generating
lower bounds. Cui et al claimed that their MIQCQP$_{1}$  is more efficient than the  CPLEX.

Now we generate the above test problems with $n$ ranged from 100 to 400. For each $n$, we generate 20 instances by selecting $n$ stocks from the 481 stocks randomly.

We illustrate the numerical results by mean of Figure \ref{fig1}, comparing the average computational time of SSAl with that of (MIQCQP$_{1}$). Figures \ref{fig1} shows that SSAL also promises (MIQCQP$_{1}$). For example, our method is 20 times faster than that of (MIQCQP$_{1}$), under the best situation.

\begin{figure}
\centering
\includegraphics[width=4in]{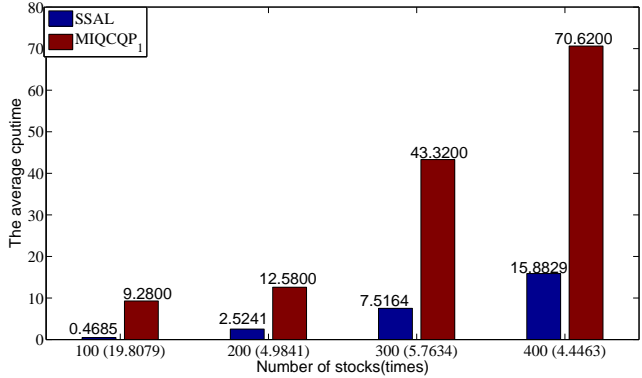}
\caption{Average CPU time (in seconds) for problems (P$_{NS}$) of SSAL and $(MIQCQP_{1})$
}\label{fig1}
\end{figure}


\subsubsection{Simulation data set}
\label{subsec:Simulation data set}

To further demonstrate the effective of performance of SSAL for large-scale problems, in this subsection, we simulate the test examples with the large size from 500 to 1000.
The simulation data is generated in the same fashion as in \cite{cui3}. The parameters in the model (P$_{NS}$) are generated randomly from the uniform distributions: ~$\alpha_{i}\in[0,0.03]$,~ $\beta_{ij}\in[0.3,2]/m$, $i=1,2,\dots n, ~j=1,2,\dots m$;~$\sigma_{ij}$, the covariance of factors $i$ and $j$ is calculated with sampled series from $[0,0.4]$ randomly, ~$i=1,2,\dots,m~,j=1,2,\dots,m$; ~ $\sigma_{\epsilon_{i}}\in[0,0.002]$,~$i=1,2,\dots n$.

Our goal is to explore the tendency of the computational time as $n$ increasing. 10 instances are generated randomly and recorded the three indexes: minimum, maximum and average numbers of iterations and computational time, respectively. It is importance to note that  CPLEX 12.6 is not able to solve the most of instances within 3600s.

 \begin{table}
 \caption{The variability for the large-scale problems}
\centerline{\small\begin{tabular}{ccccccccc}
        \hline
    \multirow{2}{*}{n}&\multicolumn{3}{c}{Computating time (s)}&&\multicolumn{3}{c}{Iteration}\\
    \cline{2-4}   \cline{6-8}\\
 & minimum & maximum& average & & minimum &maximum&average \\
    \hline
500 &     21.0029 &  48.8297&  29.82437&& 8  &  15& 12\\
600&36.9216&57.6063&47.9425&&8&11&10\\
700& 68.4558&103.8730 & 85.6297&&9 &12 &11 \\
800&126.9821 &160.2004 & 140.9647&& 8& 10&9 \\
900&173.8905 &253.4003 & 212.5805&& 8&10 &10 \\
1000&249.5224 &493.6785 & 355.7434&& 8&14 &11 \\
\hline
\end{tabular}} \label{table2}
\end{table}
It is clear from Table \ref{table2} that SSAL is robust and efficient as the number of assets increases. In addition, the number of iteration less depends on the problem size.
\subsection{Compressed sensing and subset selection}
\label{subsec2:Compressed sensing and subset selection}
In this subsection, we apply SSAL method proposed in Section $\ref{sec:Alternating direction method}$ to solve the problem (NCSB). We generate data as follows: the data matrix $A\in R^{p\times n}$ with orthonormal rows, each entry from a standard Gaussian distribution $a_{ij}\thicksim \mathcal{N}(0,1)$. The original signal $f$ is generated with $K$ nonzero entries, each sampled is the absolute value of $\hat{x}$, where $\hat{x}$ from an $\mathcal{N}(0,1)$ distribution. An observation vector $b=Af+r\in R^{p}$, where $r$ is Gaussian noise of variance $\sigma^{2}=0.01$.

Firstly, we illustrate the behavior of SSAL visually.
The $x$-update step involves an unconstrained quadratic optimization problem $P_{x}$.
The $y-$update step has exactly the same expression as in the Theorem \ref{updateyz}, where the semicontinuous bound $a_{i}=1\times 10^{-5}$, $i=1,2,\dots,n$.  We generate one instance with $(p,n)=(1024,2048)$, $K=500$ randomly.
The parameters in SSAL are set as follows: the initial Lagrangian multiplier $\lambda^{0}=0$, the tolerance parameter is set as $\epsilon=10^{-4}$, $u=10$, the step-size $\omega=1$ and the initial penalty $\rho=1$.
We use notations "Original" and "SSAL" as the original signal, the signal recovered by the SSAL respectively.
From the bottom graph of Figure \ref{fig4}, it is clear that our method almost recover the original signal exactly.

We also compare the performance of SSAL with the PD method of \cite{penal} for this simple test.
For the PD method, we set the tolerance parameter $eps=10^{-6}$ and the initial point as 0. To evaluate the quality of these sparse approximate solutions, we adopt a similar criterion as described in \cite{Figueiredo07}. The associated squared error is defined as:
\begin{equation*}
MSE:=\frac{1}{n}\|f-\hat{f}\|^{2}_{2},
\end{equation*}
where $\hat{f}$ is an estimate of the original signal $f$.
\begin{figure}
\centering
\includegraphics[width=5in,height=4in]{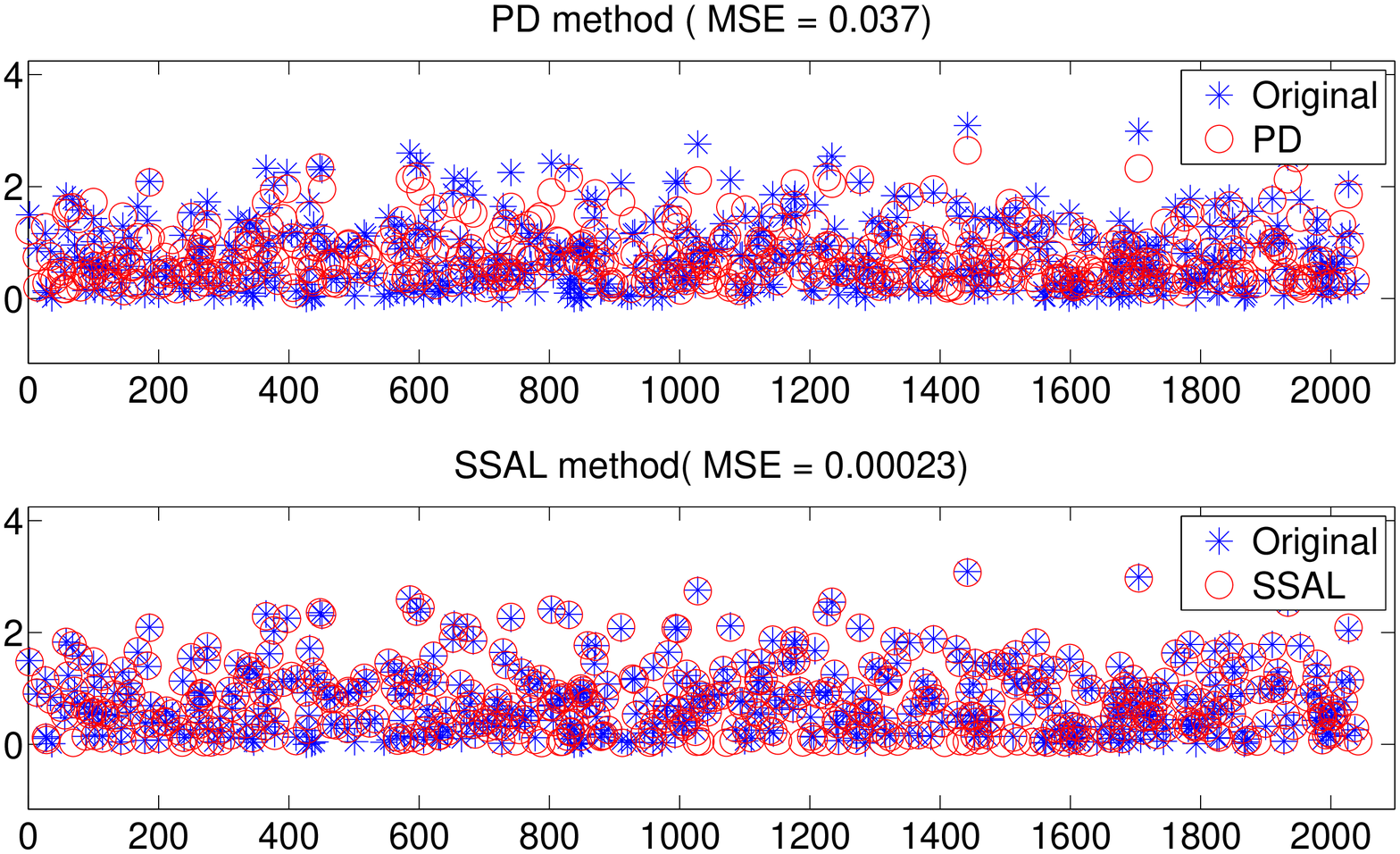}
\caption{Compare the performance of SSAL method and the PD method  to the problem (NCSB)\label{fig4}}
\end{figure}

Comparing the top graph to the bottom one in Figure \ref{fig4}, we observe
that SSAL is better recover the signal than the PD method, and the MSE of the approximate solutions obtained by SSAL are much lower than that of the PD method.

In order to illustrate that SSAL is the winner, we compare SSAL with the PD method \cite{penal} with two different sizes. We construct the testing data and then set of all the parameters’ values in the same way as the previous test.
For each size, we apply SSAL and the PD method to solve problem (NCSB) with four different values of $K$. For each such $K$, we generate the data set consisting of 50 instances randomly. The first 50 problems with $K=50$, then $K=100$, $K=150$ and $K=200$.
Figure \ref{cputime_ratio1} and Figure \ref{MSE_ratio1}  show the 10-logarithm of the ratio $t_{PD}/t_{SSAL}$ and the 10-logarithm of the ratio $MSE_{PD}/MSE_{SSAL}$ for $(p,n)=(512,1024)$, respectively.
Figure \ref{cputime_ratio2} and Figure \ref{MSE_ratio2} show the 10-logarithm of the ratio $t_{PD}/t_{SSAL}$ and the  10-logarithm of the ratio $MSE_{PD}/MSE_{SSAL}$ for $(p,n)=(1024,2048)$, respectively.
\begin{figure}
\centering
\includegraphics[width=4in]{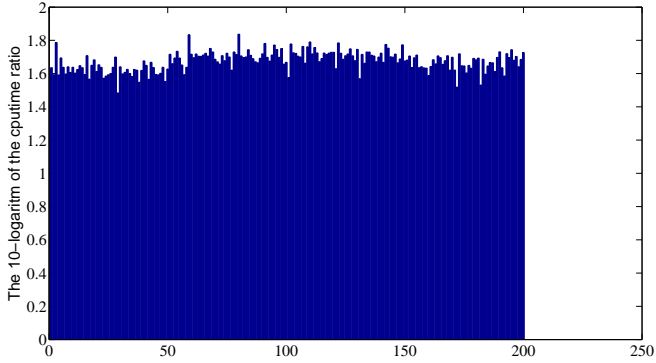}
\caption{The 10-logarithm of the ratio $\frac{t_{PD}}{t_{SSAL}}$ for $(p,n)=(512,1024)$ \label{cputime_ratio1}}
\end{figure}
\begin{figure}
\centering
\includegraphics[width=4in]{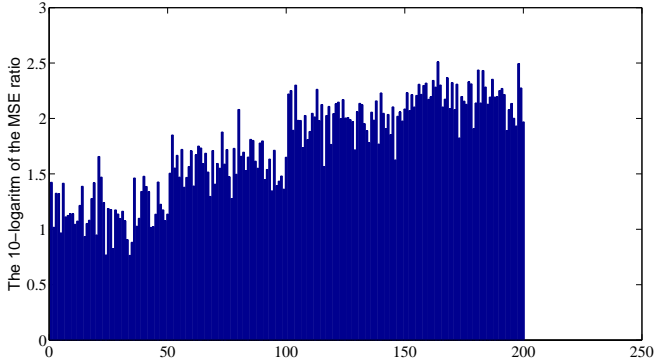}
\caption{The 10-logarithm of the MSE $\frac{MSE_{PD}}{MSE_{SSAL}}$ for $(p,n)=(512,1024)$\label{MSE_ratio1}}
\end{figure}
\begin{figure}
\centering
\includegraphics[width=4in]{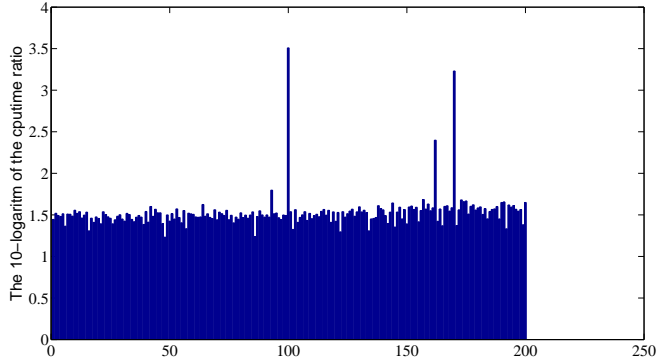}
\caption{The 10-logarithm of the ratio $\frac{t_{PD}}{t_{SSAL}}$ for $(p,n)=(1024,2048)$ \label{cputime_ratio2}}
\end{figure}
\begin{figure}
\centering
\includegraphics[width=4in]{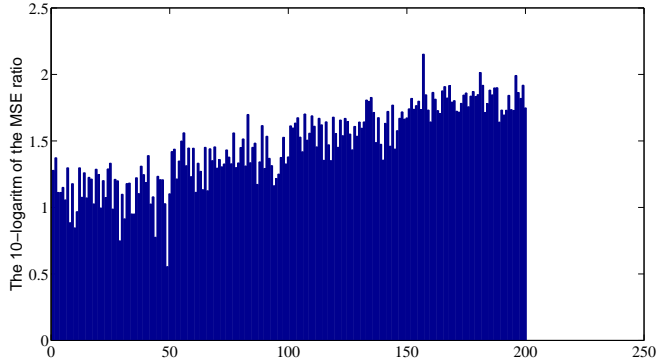}
\caption{The 10-logarithm of the MSE $\frac{MSE_{PD}}{MSE_{SSAL}}$ for $(p,n)=(1024,2048)$\label{MSE_ratio2}}
\end{figure}

From Figure \ref{cputime_ratio1} and Figure \ref{cputime_ratio2}, we observe that SSAL is substantially faster than the PD method for every case. Moreover, our method is 40 times faster than the PD method on average.
For the best situation, our method is more than 100 times faster than the PD method.
In addition, SSAL outperforms the PD method in terms of solution quality since it
archives much smaller MSE values. Thus, SSAL much better recovers the original signal than the PD method.

Empirical study demonstrates that SSAL is a powerful method to solve problem (P) for the following reasons:
(1) SSAL can find the good enough approximate solution of problem (P).
For example, the optimal objective value obtained by the SSAL and the CPLEX are almost the same for the problem (P$_{NS}$). Furthermore, the quality of the approximate solution of problem (NCSB) obtained by SSAL is generally better than that of the PD method.
(2) SSAL is a very fast method. For instance, SSAL is 200 times faster than the CPLEX method and 5 times faster than the (MIQCQP$_{1}$) method of \cite{cui} for the problem (P$_{NS}$). For the (NCSB) problem, SSAL is 40 times faster than the PD method of \cite{penal}.
\section{Conclusion}
\label{sec:Conclusion}
We have proposed a splitting and successively solving augmented Lagrangian (SSAL) method. Subproblems  have been decomposed by fixing certain variables and solving them at each iteration  alternatively.
The convergence of SSAL has  been proved, under some suitable assumptions.
Real world data and simulation study show that SSAL outperforms the CPLEX 12.6, the (MIQCQP$_{1}$) reformulation \cite{cui} and the PD method \cite{penal}, while enjoying a similar cardinality of decision variable. In particularly, SSAL is powerful when the size of the problem increase largely.


%

%

\small
\bibliographystyle{plain}
\bibliography{OMS-Liang}

\end{document}